\documentclass[a4paper]{article}
\usepackage[utf8]{inputenc}

\title{Weak and Strong Extremal Biquadratics}
\author{Grigoriy Blekherman, Bogdan Rai\cb{t}\u{a}, Isabelle Shankar, and Rainer Sinn}
\date{\today}
\usepackage{combelow}
\usepackage{a4wide}
\usepackage{amssymb}
\usepackage{amsmath}
\usepackage{euscript}
\usepackage{amsthm}
\usepackage{mathtools}
\usepackage{amsopn}
\usepackage{stackrel}
\usepackage[all]{xy}
\usepackage{setspace}
\usepackage{float}
\usepackage{blkarray}
\usepackage[pdftex, colorlinks, linkcolor=blue, citecolor=blue, urlcolor=blue]{hyperref}
\usepackage{pgfplots}
\usepackage[margin=1cm]{caption}
\usepackage{subcaption}
\pgfplotsset{width=7cm, compat=1.10}
\usepgfplotslibrary{fillbetween}
\usepackage[ruled,vlined]{algorithm2e}
\usetikzlibrary{matrix}
\usepackage{mathdots}
\usepackage{xfrac}

\usepackage{todonotes}

\usepackage{mathtools} 
\mathtoolsset{showonlyrefs,showmanualtags}
\numberwithin{equation}{section}
\usepackage{amsmath}

\newcommand{\R}{\mathbb{R}}
\newcommand{\Z}{\mathbb{Z}}
\newcommand{\C}{\mathbb{C}}

\newcommand{\ProjP}{\mathbb{P}}
\renewcommand{\P}{\mathbb{P}}
\newcommand{\V}{\mathcal{V}}

\DeclareMathOperator{\pic}{Pic}
\newcommand{\dif}{\mathrm{\;d}}

\newcommand{\wh}{\widehat}
\newcommand{\ratto}{\dashrightarrow}

\newtheorem{theorem}{Theorem}[section]
\newtheorem{cor}[theorem]{Corollary}
\newtheorem{lemma}[theorem]{Lemma}
\newtheorem{prop}[theorem]{Proposition}

\theoremstyle{definition}
\newtheorem{definition}[theorem]{Definition}

\newtheorem{conjecture}[theorem]{Conjecture}

\theoremstyle{remark}

\newcommand{\qnm}{\mathrm{Q}_{n,m}}
\newcommand{\qn}{\mathrm{Q}_{n}}
\newcommand{\bnm}{\mathrm{B}_{n,m}}
\newcommand{\bn}{\mathrm{B}_{n}}
\newcommand{\qqnm}{\mathrm{QQ}_{n,m}}
\newcommand{\qqn}{\mathrm{QQ}_{n}}
\newcommand{\pbnm}{\mathrm{PB}_{n,m}}
\newcommand{\pbn}{\mathrm{PB}_{n}}

\newcommand{\CC}{\mathbf{C}}
\newcommand{\TT}{\mathbf{T}}
\newcommand{\SSS}{\mathbf{S}}

\begin{document}

\maketitle
\begin{abstract}
\noindent 
We study quasiconvex quadratic forms on $n \times m$ matrices which correspond to nonnegative biquadratic forms in $(n,m)$ variables. We disprove a conjecture stated by Harutyunyan--Milton (Comm. Pure Appl. Math. 70(11), 2017) as well as Harutyunyan--Hovsepyan (Arch. Ration. Mech. Anal. 244, 2022) that extremality in the cone of quasiconvex quadratic forms on $3\times 3$ matrices can follow only from the extremality of the determinant of its acoustic tensor, using previous work  by Buckley--\v{S}ivic (Linear Algebra Appl. 598, 2020).
Our main result is to establish a conjecture of Harutyunyan--Milton (Comm. Pure Appl. Math. 70(11), 2017) that weak extremal quasiconvex quadratics on $3 \times 3$ matrices are strong extremal.  
Our main technical ingredient is a generalization of the work of Kunert--Scheiderer on extreme nonnegative ternary sextics (Trans. Amer. Math. Soc. 370(6), 2018).  Specifically, we show that a nonnegative ternary sextic, which is not a square, is extremal if and only if its  variety (over the complex numbers) is a rational curve and all its singularities are real.
\end{abstract}

\section{Introduction}
Composite materials have been studied mathematically since at least the first half of the nineteenth century, receiving historical contributions from Poisson, Faraday, Rayleigh, Maxwell, and Einstein. Basic ideas describing effective properties of composite materials were introduced in the early twentieth century by Voigt \cite{Voigt} and Reu{\ss} \cite{Reuss}, and were subsequently refined by Hill \cite{hill1952elastic} and Hashin--Strikman \cite{hashin1962variational,hashin1963variational}. These works of resounding impact in the Material Sciences literature were subsequently embedded in the Analysis of Partial Differential Equations. Here we recall the work from the late twentieth century of Murat and Tartar \cite{MuratTartar,tartar1979,tartar1985}, Ball \cite{ball77}, Lurie--Cherkaev \cite{lurie1984exact}, Francfort--Murat \cite{francfort1986homogenization}, Kohn \cite{Kohn}, Allaire--Kohn \cite{allaire1993optimal,allaire1994optimal}, and Milton \cite{MiltonCPAM}, to name a few. The fundamental mathematical theory developed around the practical questions in material science is comprehensively summarized in the modern books by Milton \cite{MiltonBook} and Grabovsky \cite{GrabovskyBook}.

To conclude this brief historical discussion, we zoom in on the remarkable observation of Milton \cite{MiltonCPAM} that a special class of \emph{extremal quasiconvex translations}, which we will refer to as \emph{weak extremals} in the remainder of the text, can be used to retrieve and perhaps improve the celebrated Hashin--Strikman bounds \cite{hashin1962variational,hashin1963variational} by the method of Murat--Tartar \cite{MuratTartar} and Lurie--Cherkaev \cite{lurie1984exact} (see Appendix \ref{sec:composites} for an exposition of the origin of Milton extremals). Of the advances of the twenty first century in this direction, we highlight the fact that weak extremals have a Krein--Milman type property: namely, any quasiconvex quadratic form can be decomposed as a sum of a weak extremal and a convex form. While the relevance of quasiconvex quadratic forms for variational problems in Continuum Mechanics was recognized a long time ago \cite{morrey52,vanhovezbMATH03046021}, their separation from convex forms is not very well understood. One useful viewpoint is to identify the \emph{quasiconvex quadratic forms} with \emph{nonnegative biquadratic forms} by quotienting out null Lagrangians (see Section \ref{sec:prel} for definitions and details). Under this identification, weak extremals are nonnegative biquadratic forms that are completely separated from \emph{sums of squares}, making the tools and advances of Real Algebraic Geometry available for this study \cite{zbMATH07102076,GGMzbMATH06572970,BS2020,KS2018,Quarez2015} (see \cite{convex_algebraic_geometry,real_alg_geom_BCR} for general background in Real Algebraic Geometry). 

The results of Harutyunyan--Milton \cite{HMcalcvar,HMarma,HarutyunyanMiltonCPAM} lead to the declared goal:
\begin{quote}
    Classify all $3\times 3$ weak extremal quasiconvex quadratic forms.
\end{quote}
A clarification of the meaning of extremality is in order. In light of the Krein--Milman property for weak extremals \cite[Sec. 6]{MiltonKM}, it seems natural to make the comparison with the classical notion of extreme rays from Convex Analysis. Indeed, an extreme ray of the cone of nonnegative biquadratics is also a weak extremal, except for the trivial case of a perfect square. The opposite question, whether weak extremals are extreme rays, is far from obvious and was asked in \cite{HarutyunyanMiltonCPAM}. 

We now make the connection between this question and Real Algebraic Geometry precise.  Let $\pbnm$ denote the proper convex cone of biquadratic forms in $(n,m)$ variables. Contained inside $\pbnm$ is the smaller cone, $\Sigma_{n,m}$, of biquadratic sums of squares.  A strong extremal is just an extreme ray of $\pbnm$. That is, $f$ is a strong extremal of $\pbnm$ if whenever we subtract a nonnegative biquadratic (which is not a scalar multiple of $f$) from $f$ we leave the cone $\pbnm$. Likewise, a nonnegative biquadratic is weak extremal if subtracting a form in $\Sigma_{n,m}$ moves it outside $\pbnm$, a slightly weaker notion of extremality.  For a precise formulation we refer the reader to Definitions \ref{def:SE}, \ref{def:WE}.

We answer in the positive a question of Harutyunian--Milton \cite[Conj. 2.8]{HarutyunyanMiltonCPAM},  \cite[Conj. 2.7]{harutyunyan2021extreme}:
\begin{theorem}\label{thm:mintro}
Weak extremal quasiconvex quadratic forms on $\R^{3\times 3}$ are strong extremals.
\end{theorem}

To ensure better access to our work to readers specialized in either Partial Differential Equations or Algebraic Geometry, we also state the equivalent formulation of our main result:
\begin{theorem}\label{thm:mintro2}
A weak extremal nonnegative biquadratic on $\R\P^2\times \R\P^2$ (which is the same as the restriction of a quadratic form to the set of $3\times 3$ real matrices of rank $1$) spans an extreme ray of the convex cone of nonnegative quadratics (in $(3,3)$ variables).
\end{theorem}

We show below that this means that every face of the cone of nonnegative biquadratics in $(3,3)$-variables of dimension at least $2$ contains a square. We recast Theorems \ref{thm:mintro} and \ref{thm:mintro2} as Corollary \ref{cor:main} and Theorem \ref{thm:main}. The latter is proved in Section \ref{sec:weak-strong}, using results from Section \ref{sec:ternary_sextics}.

In Section \ref{sec:counterexample}, we disprove a conjecture of Harutyunyan--Milton \cite[Conj. 2.8]{HarutyunyanMiltonCPAM} and Harutyunyan--Hovsepyan \cite{harutyunyan2021extreme}, 
that extremality in the cone of quasiconvex quadratic forms on $3\times 3$ matrices can follow only from the extremality of the determinant of its acoustic tensor. Our example is based on the work of Buckley--\v{S}ivic \cite{BS2020}, using results of Kunert--Scheiderer \cite{KS2018}.

\begin{theorem}\label{thm:main-ext_not_ext_det}
There exists a (weak) extremal quasiconvex quadratic form on $\R^{3\times 3}$ such that the determinant of its acoustic tensor is not an extreme ray in the cone of nonnegative ternary sextics.
\end{theorem}

The \emph{acoustic tensor} (or \emph{$y$-matrix}) of a quadratic form $\Phi$ on $\R^{3\times 3}$ is the matrix $T(y)$ such that $\Phi(x\otimes y)=\langle x,T(y)x\rangle$. Equivalently, if $F$ is a biquadratic form then $F$ can be written as $x^TT(y)x$ where $T(y)$ is a matrix of quadratic forms in the variables $y$.

We can view our main results in a complementary way. Theorem \ref{thm:mintro} states that weak extremals in the physically relevant case of three dimensions are simply strong extremals, and thus do not form a larger class. On the other hand, Theorem \ref{thm:main-ext_not_ext_det} shows that weak extremals do form a larger class than forms whose extremality can be deduced from extremality of their acoustic tensor. 

We alluded to the fact that dimension three is special for our study, a fact that deserves more attention. Indeed, in dimension two, nonnegative biquadratics are sums of squares \cite{calderon,terpstra}, so there are no weak extremals in this case. That positive biquadratics on $\R^3\times \R^3$ which are not sums of squares exist was already known to Terpstra \cite{terpstra} and explicit examples can be found in both the Algebraic Geometry literature \cite{BS2020,Choi,choilamzbMATH03540962} and the Analysis literature \cite{HMcalcvar,serre1,serre2}. Among these, explicit examples of nontrivial extremals were presented in \cite{BS2020,choilamzbMATH03540962,HMcalcvar}. 
Finally, dimension three is special also because in dimension four or higher there are weak extremals that are not strong, as was recently proved by Harutyunyan--Hovsepyan \cite[Rk. 2.6]{harutyunyan2021extreme}. Therefore our main result does not hold in higher dimensions.

\medskip
\textbf{Acknowledgements.} We would like to thank Kristian Ranestad for very helpful discussions. We also thank Felix Otto and Bernd Sturmfels for pointing out the work on the relation between the Theory of Composites and Algebraic Geometry. We thank Davit Harutyunian for interesting discussions at an early stage of this project. GB is partially supported by US National Science Foundation grant  DMS-1901950. 

\section{Quasiconvexity and biquadratics}\label{sec:prel}

Motivated by examples in Continuum Mechanics (see \cite{GrabovskyBook,MiltonBook} and Appendix \ref{sec:composites}), we investigate quadratic forms defined on matrices
$$
\qnm\coloneqq \{\Phi\colon\R^{n\times m}\rightarrow{\R}\colon \Phi\text{ quadratic form}\}.
$$
These can be written as $\Phi(Z)=\langle Z, CZ\rangle$, where $C$ is a symmetric linear map on $\R^{n\times m}$ and we use the standard inner product $\langle A,B \rangle = \mathrm{trace}(AB^T)$ on matrices. This section touches on \cite[Sec. 5.3.2]{Dacorogna}, where other related facts and more references can be found. 

By the Spectral Theorem, we have the following:
\begin{lemma}\label{lem:convex}
Let $\Phi\in\qnm$. The following are equivalent:
\begin{enumerate}
    \item $\Phi(Z)\geq 0$ for all $Z\in\R^{n\times m}$;
    \item $\Phi$ is convex;
    \item $\Phi$ is a sum of squares, i.e., $\Phi(Z)=\sum_{i=1}^{nm}\langle Z, M_i\rangle^2$  for some matrices $M_i\in\R^{n\times m}$.
\end{enumerate}
\end{lemma}
In fact, for this result, we do not use the special structure of the linear space of matrices. However, convex functions seldom appear in applications; a broader class is that of \textbf{polyconvex} functions \cite{ball77,morrey52}, which are defined as convex functions of the minors, i.e.,
$$
\Phi(Z)=g(\mathbf{M}(Z)),
$$
where $g$ is convex and $\mathbf M(Z)$ is the vector of all minors of $Z$. 
For example, $2\times 2$ minors (and linear combinations thereof) are polyconvex functions in $\qnm$. In fact, polyconvex quadratic forms are easily characterized:
\begin{prop}[\cite{vanhovezbMATH03046021}]
\label{prop:polyconvex}
A form $\Phi\in\qnm$ is polyconvex if and only if it can be written as the sum of a convex form plus a linear combination of $2\times 2$ minors, i.e.,
$$\Phi(Z)=\sum_{i=1}^{nm}\langle Z, M_i\rangle^2+\sum_{\deg \mathcal M=2}c_{\mathcal M} \mathcal M(Z),$$
where $c_\mathcal{M}\in \R$ and the sum runs over the $2\times 2$ minors $\mathcal M$ of $Z$.
\end{prop}

Moreover, as far as variational problems are concerned, the class of interest is smaller, {as we will explain below}. This is the class of so-called \textbf{quasiconvex} functions, which are defined by the minimization property that
$$
0\leq \int_{\R^m} \Phi(Z+\nabla \varphi(x))-\Phi(Z)\mathrm{\;d} x
$$
for all $Z\in\R^{n\times m}$ and all test functions $\varphi\in C_c^\infty(\R^m,\R^n)$. This is to say that among all functions with affine ``boundary conditions'' $u(x)= Zx$ outside a bounded set $B$, 
this affine deformation itself is a minimizer of the energy $\int_B \Phi(\nabla u(x))\mathrm{\;d}x$.

While quasiconvexity is very difficult to describe in general, we have a simple characterization of quasiconvex functions in the class of quadratic forms:
\begin{lemma}[{\cite{vanhovezbMATH03046021}}]\label{lem:qc_r1c}
Let $\Phi\in\qnm$. The following are equivalent:
\begin{enumerate}
    \item\label{itm:a} $\Phi$ is quasiconvex;
    \item\label{itm:b} $\Phi(x\otimes y)\geq 0$ for all $x\in\R^n,\,y\in\R^m$;
    \item\label{itm:c} $\Phi$ is convex in rank one directions.
\end{enumerate}
\end{lemma}
In general, functions satisfying the equivalent conditions in Lemma \ref{lem:qc_r1c} are called \textbf{rank-one convex}. It follows from the lemma that, in the class of quadratic forms, quasiconvexity reduces to rank-one convexity, which is a local notion that is, in theory, computable. In fact, it is a well known fact that quasiconvexity implies rank-one convexity in general, but the converse is an exceptionally difficult problem raised in \cite{morrey52} and partially settled in the negative in \cite{Sverak,Grabovski}\footnote{Remarkably, the example in \cite{Grabovski} is constructed by appealing to the Theory of Exact Relations coming from the study of composite materials, cf. \cite{GrabovskyBook,MiltonBook}.}, with the $2\times 2$ case remaining open to this day \cite{FaracoSzekelyhidi,HKL}. 

Lemma \ref{lem:qc_r1c} and Proposition \ref{lem:char_polyconvex} make it plain that polyconvex quadratic forms are quasiconvex. Also, the scalar cases $n=1$ or $m=1$ are completely uninteresting because all these semiconvexity notions collapse to convexity.
\begin{proof}[Proof of Lemma \ref{lem:qc_r1c}]
To see that \ref{itm:b} is equivalent to \ref{itm:c}, note that, by differentiating twice at $Z$ in the rank-one direction $x\otimes y$, we obtain the equivalent formulation of \ref{itm:c} as
$$
\langle x\otimes y, \nabla^2 \Phi(Z) [x\otimes y] \rangle\geq 0.
$$
However, the Hessian of $\Phi$ is just the linear map $C$, so \ref{itm:c} is equivalent to
$$
\langle x\otimes y, C (x\otimes y) \rangle\geq 0,
$$
which is \ref{itm:b}. 

Next, we show that \ref{itm:b} implies \ref{itm:a}. Note that 
\begin{align*}
    \int_{\R^m} \Phi(Z+\nabla \varphi(x))-\Phi(Z)\mathrm{\;d} x&=  \int_{\R^m} \langle Z,C\nabla \varphi(x)\rangle +\langle CZ,\nabla \varphi(x)\rangle + \langle \nabla \varphi(x),C\nabla \varphi(x)\rangle \dif x\\
    &=\int_{\R^m}  \langle \nabla \varphi(x),C\nabla \varphi(x)\rangle \dif x\\
     &=\int_{\R^m}  \langle \widehat{\nabla \varphi(\xi)},C\overline{\widehat{\nabla \varphi(\xi)}}\rangle \dif \xi\\
      &=\int_{\R^m}  \langle  \hat \varphi(\xi)\otimes \xi,C \overline{\hat{\varphi}(x)}\otimes \xi \rangle \dif \xi\\
      &=\int_{\R^m}  \Phi(  \Re\hat \varphi(\xi)\otimes\xi)+ \Phi(\Im{\hat{\varphi}(x)}\otimes \xi) \dif \xi\geq0,
\end{align*}
where the first equality follows by integration by parts, the second equality is given by Plancherel's Theorem. The fourth equality follows since the Fourier Transform of a derivative obeys $\widehat{\partial_k f}(\xi)=-\xi_k\mathrm i \hat f(\xi)$, up to a scalar multiple which depends on normalization. The last equality then follows by computation. This concludes the proof of \ref{itm:a}.

Finally, the fact that quasiconvexity implies rank-one convexity is true in general, see the classical reference \cite[Thm. 5.3(i)]{Dacorogna}, so we also have that \ref{itm:a} implies \ref{itm:c}.
\end{proof}
The above proof shows that the link between the gradients that appear in the definition of quasiconvexity and rank-one matrices is given by the Fourier Transform.

By writing 
\begin{equation}\label{eq:def_bq}
F(x,y)\coloneqq \Phi(x\otimes y),
\end{equation}
where $\Phi\in\qnm$, we obtain a \textbf{biquadratic form}.  That is a form that is homogeneous of degree $2$ in $x$ and simultaneously homogeneous of degree $2$ in $y$. We define 
$$
\bnm\coloneqq \{F\colon\R^{n}\times \R^m\rightarrow \R\colon F\text{ biquadratic form}\}.
$$
In view of Lemma \ref{lem:qc_r1c}, we would like to identify quasiconvex quadratic forms with nonnegative biquadratics. To this end, we define
\begin{align*}
    \qqnm \coloneqq \{\Phi\in\qnm\colon \Phi\text{ quasiconvex}\}, \qquad
    \pbnm \coloneqq \{F\in\bnm\colon F\geq 0\}.
\end{align*}
We will also use the abbreviations $\qqnm=\qqn$ and $\pbnm=\pbn$, as well as $\qnm=\qn$ and $\bnm=\bn$ when $n=m$. We have the following (see \cite[Ex. 2.11]{HarriszbMATH00052497}):

\begin{lemma}\label{lem:sigma}
The map $\sigma\colon \qnm \to \bnm$, $\Phi\mapsto F = \Phi(x\otimes y)$, is a linear surjection. Its kernel is the linear space spanned by the $2\times 2$ minors of $\mathcal{M}$.
Moreover, $\sigma(\qqnm)=\pbnm$. 

Algebraically, this map $\sigma$ is the canonical quotient map from the degree $2$ part of the coordinate ring of $\R^{m\times n}$, which is the vector spaces of homogeneous polynomials of degree $2$ in the entries of the matrices, to the (degree $2$ part of the) coordinate ring of the Segre embedding $\P^{n-1}\times \P^{m-1}$ whose ideal is generated by the $2\times 2$ minors.
\end{lemma}
We will use the term \textbf{null Lagrangian} when referring to linear combinations of minors. The previous lemma says precisely that the null Lagrangians coincide with the ideal of the Segre embedding in degree $2$. This does not hold in higher degrees (even the above characterization of polyconvex functions Proposition \ref{prop:polyconvex} does not generalize). 

The null Lagrangians are precisely the functions $\Phi$ such that $\pm \Phi$ is quasiconvex, see \cite[Thm. 5.20]{Dacorogna}. In particular, the kernel of the map $\sigma$ above consists of homogeneous null Lagrangians of degree two.

By writing a biquadratic $F$ in the form
$$
F(x,y)=\sum_{ijkl}x_iy_jC_{ijkl}x_ky_l=\sum_{jk}y_j\underbrace{\left(\sum_{ik}x_iC_{ijkl}x_k\right)}_{T(x)_{kl}}y_l
$$
and imposing the symmetries 
$$
C_{ijkl}=C_{ilkj}=C_{kjil}
$$
we ensure that both the $x$-matrix $T(x)$ and the analogously defined $y$-matrix of $F$ are symmetric. It is also the case that the quadratic form on matrices defined by
$$
\Phi(Z)\coloneqq\sum_{ijkl}Z_{ij}C_{ijkl}Z_{kl}
$$
is such that $\sigma(\Phi)=F$. In fact, $\tau\colon F\mapsto \Phi$ thus defined is a right inverse of $\sigma$: if $\sigma(\Psi)=G$, then $\Psi$ differs by $\tau(G)$ only by a null Lagrangian. These considerations prove Lemma \ref{lem:sigma}.

In light of these considerations it is easy to characterize polyconvex quadratic forms:
\begin{lemma}\label{lem:char_polyconvex}
A form $\Phi\in\qnm$ is polyconvex if and only if $\sigma(\Phi)$ is a sum of squares (SOS).
\end{lemma}

As discussed in Appendix~\ref{sec:composites}, we are interested in extremal elements of {$\qqnm$}, notions that we will now clarify. In particular, we will now distinguish between \textbf{strong extremals}  and \textbf{weak extremals}, which we  define following \cite{HMcalcvar}:
\begin{definition}\label{def:SE}
Let $\Phi\in\qnm$. We say that $\Phi$ is a \textbf{strong extremal} of $\qqnm$ if the ray $\R_+ \sigma(\Phi)$ spanned by $\sigma(\Phi)$ is an extreme ray of the convex cone $\pbnm$. 
\end{definition}
We will also refer to extreme rays $F$ of $\pbnm$ as \textbf{strong extremals}. Recall that this means that if we can write $F=F_1+F_2$ for $F_1,\,F_2\in\pbnm$, then $F_1$ and $F_2$ are linearly dependent.
\begin{definition}\label{def:WE}
Let $\Phi\in\qnm$ (resp. $F\in\bnm$). We say that $\Phi$ (resp. $F$) is a \textbf{weak extremal} of $\qqnm$ (resp. $\pbnm$) if
$$
\Phi(x\otimes y)\geq \langle x,My\rangle^2\qquad(\text{resp. }F(x,y)\geq \langle x,My\rangle^2)\qquad\text{ for all }x\in\R^n,y\in\R^m
$$
implies that $M\in\R^{n\times m}$ is zero.
\end{definition}
In other words, $\Phi$ is a weak extremal of $\qqnm$ if and only if $\sigma(\Phi)$ is a weak extremal of $\pbnm$. On the other hand,  we can describe the relation between strong extremals of $\qqnm$ and $\pbnm$ very neatly using classic concepts of convex analysis. We have that $\ker \sigma $ is the \emph{lineality space} of $\qqnm$ and that $\tau$ is a linear map which induces an isomorphism of \emph{pointed cones} $\mathrm{im\,}\tau$ and $\pbnm$:
$$
\qqnm=\ker\sigma+\mathrm{im\,}\tau,\quad\tau\colon\pbnm\rightarrow\mathrm{im\,}\tau\text{ linear isomorphism}.
$$
Therefore, by the Krein--Milman theorem in the pointed cone $\pbnm$, all quasiconvex quadratic forms can be written as a null Lagrangian plus a sum of \emph{extreme rays} of the cone of nonnegative biquadratics (under the explicit map $\tau$). 

Weak extremals have a special Krein--Milman property of their own:
\begin{theorem}[{\cite[Sec. 6]{MiltonKM}}]
Any quasiconvex quadratic form can be written as the sum of a weak extremal (of $\qqnm$) and a (poly)convex quadratic form.
\end{theorem}
Equivalently, in the language of biquadratics we have, cf. Lemma \ref{lem:char_polyconvex}:
\begin{cor}
A nonnegative biquadratic can be written as a weak extremal  plus a sum of squares.
\end{cor}

If either $n=2$ or $m=2$, the structure of weak extremals of $\qqnm$ is trivial (linear combination of $2$-minors), as follows from the old result of Terpstra \cite{terpstra} (see also Calder\'on \cite{calderon}). We can also infer that strong extremals are just perfect squares. It also follows that quasiconvex quadratic forms are polyconvex in this case, a fact that is no longer true if $n=m=3$, as proved in the same work of Terpstra \cite{terpstra}.

It is easy to see that if $\Phi\in\qqnm$ is a strong extremal, then it is either  a weak extremal or a perfect square. Conversely, in the cases $n\leq 2$ or $m\leq 2$, the only weak extremals are null Lagrangians, so there are no nontrivial examples of extremals (of either kind). On the other hand, for $n=m=3$, the classical example of Choi and Lam from 1977 \cite{choilamzbMATH03540962} shows that nontrivial strong extremals of $\mathrm{PB}_3$ exist. In higher dimensions $n,m\geq4$, a recent example of Harutyunian and Hovsepyan from 2021 shows that weak extremals can fail to be strong extremals in $\mathrm{PB}_n$. Here, we focus on the three dimensional case $n=m=3$ and prove the following striking result, that gives a positive answer to a question of Harutyunian--Milton \cite{HarutyunyanMiltonCPAM}, formulated clearly in the first part of \cite[Conj. 2.7]{harutyunyan2021extreme}:
\begin{theorem}\label{thm:main}
Weak extremals of $\mathrm{PB}_3$ are extreme rays.
\end{theorem}
Therefore, we can also infer the consequence for quadratic forms on matrices:
\begin{cor}\label{cor:main}
Weak extremals of $\mathrm{QQ}_3$ are strong extremals.
\end{cor}
The second part of \cite[Conj. 2.7]{harutyunyan2021extreme}, which was first stated in \cite[Conj. 2.8]{HarutyunyanMiltonCPAM}, asks if for nontrivial weak extremals in dimension 3, we also have that the determinant of the $x$-matrix $T(x)$ is extremal in the cone of nonnegative sextics. We disprove this assertion in Section \ref{sec:counterexample}.

\section{Characterization of weak extermals}\label{sec:weak_extremal_char}

In this section, we give alternative views on the definition of weak extremal quasiconvex functions in terms of convex and real algebraic geometry.
We will show that a quasiconvex quadratic $\Phi$ is weak extremal if and only if the zeroes (considered as matrices of rank $1$) of the associated biquadratic $F = \Phi(x\otimes y)$ span their ambient space of $n\times m$ matrices. 
To make this precise, for a biquadratic $F$ on $\P^{n-1}\times \P^{m-1}$ we define the subspace $L_F$ of real $n\times m$ matrices as 
$$L_F = \mathrm{span} \{xw^T+vy^T\colon (x,y) \in \ProjP^{n-1 }\times \ProjP^{m-1},\, F(x,y) = 0,\,  (v,w) \in \ker F''(x,y) \}$$
where $F''$ is the Hessian of $F$.
Since the form $F$ is bihomogeneous in $x$ and $y$, the kernel of the Hessian of $F$ always contains the vectors $(x,0)$ and $(0,y)$ (by Euler's identity for homogeneous polynomials). In particular, it is at least $2$-dimensional.

\begin{lemma}\label{lemma:difference_two_nonnegative_quadratics}
Let $Q_1$ and $Q_2$ be nonnegative quadratics such that $Q_1(x)=0$ implies $Q_2(x)=0$.  Then there exists a real number $\alpha > 0$ such that $Q_1 \geq \alpha Q_2$.
\end{lemma}
\begin{proof}
The quadratics can be written as $Q_1 = x^TM_1x$ and $Q_2 = x^TM_2x$ for positive semidefinite matrices $M_1$ and $M_2$. We have that $Q_i(x)=0$ if and only if $M_ix=0$. In particular, we can write $y=x-\mathrm{proj}_{\ker M_1}x$ and note that then $Q_i(x)=Q_i(y)$, so to check that $Q_1$ dominates $Q_2$, it suffices to restrict to $(\ker M_1)^\perp$, where $Q_1$ is \emph{strictly positive definite}. 

Let $\lambda_1>0$ be the minimal \emph{nonzero} eigenvalue of $M_1$ (if none exist, we have $Q_1=Q_2=0$) and $\lambda_2$ be the largest eigenvalue of $M_2$.
If $\lambda_2=0$, then $Q_2=0$, so the claim is satisfied trivially.
Else, let $\alpha \coloneqq \lambda_1/\lambda_2>0$, so 
$$
Q_1(x)=Q_1(y)\geq \lambda_1 |y|^2= \alpha\lambda_2|y|^2\geq \alpha Q_2(y)=\alpha Q_2(x),
$$
which concludes the proof.
\end{proof}

\begin{theorem}
\label{thm:char_we}
For a quasiconvex function $\Phi \in \qqnm$, let $F = \Phi( x\otimes y)$ be the associated nonnegative biquadratic.  Then $\Phi$ is weak extremal if and only if $\dim(L_F) = nm$.
\end{theorem}
\begin{proof}
Assume that dim $L_F = nm$.  Suppose that there exists $M\in\R^{n\times m}$ such that 
$$F(a,b) \geq \langle a,Mb\rangle^2 = \langle M,ab^T\rangle^2\eqqcolon G(a,b)\quad \text{for all }a,\,b,$$
where we use the usual matrix inner product $\langle A,B\rangle = \mathrm{trace}(AB^T)$ for matrices $A,B\in\R^{n\times m}$.
We claim that $M=0$. Let $(x,y)\in\P^{n-1}\times\P^{m-1}$ be a root of $F$, so that $G(x,y)=0$.

We have the Taylor expansion in direction $(v,w)$:
$$
F(x+v,y+w)=F(x,y)+\langle F'(x,y),(v,w)\rangle+\tfrac{1}{2}\langle F''(x,y)(v,w),(v,w)\rangle+ \langle F'(v,w),(x,y)\rangle+F(v,w),
$$
which follows since the Taylor formula of a polynomial is exact and the  formula is symmetric in $(x,y)$ and $(v,w)$. Since $F$ is nonnegative, we have that $F'(x,y)=0$, so that
$$
F(x+v,y+w)=\tfrac{1}{2}\langle F''(x,y)(v,w),(v,w)\rangle+o(|(v,w)|^3)\quad\text{as }|(v,w)|\to 0.
$$
Of course, $G$ must have the analogous expansion, so we can infer that if $F''(x,y)(v,w)=0$, then
$$
0=\tfrac{1}{2}\langle G''(x,y)(v,w),(v,w)\rangle=\tfrac{1}{2}\tfrac{\dif }{\dif t}\big|_{t=0} G(x+tv,y+tw)=\langle M,xv^T+wy^T\rangle^2,
$$
where the last equality follows by direct calculation. It follows that $M$ is orthogonal to a space of full dimension, which proves the claim.

Suppose dim $L_F < nm$.  Choose a nonzero matrix $M$ that is perpendicular to $L_F$. We claim that there exists a constant $\alpha > 0$ such that 
$$F(a,b) \geq \alpha \langle M,ab^T\rangle^2\quad\text{for all }a,\,b$$
and thus $\Phi$ is not a weak extremal.
As before, let $G(a,b) \coloneqq \langle M,ab^T\rangle^2$. We first verify that $F(x,y)=0$ implies $G(x,y)=0$.   

To see this, note that from the Taylor expansion of $F$ around $(x,y)$ and in direction $(x,y)$ we have 
$$
0=F(2x,2y)=\tfrac{1}{2}\langle F''(x,y)(x,y),(x,y)\rangle,
$$
and we can use the positive definiteness of $F''(x,y)$ to infer that $(x,y)\in \ker F''(x,y)$. Therefore we find that $xy^T\in L_F$, so that $G(x,y)=0$ by definition of $G$.

We next recall that at the zeroes $(x,y)$ of $F$ with $|x|=|y|=1$  we have the expansion
$$
F(x+v,y+w)=\tfrac{1}{2}\langle F''(x,y)(v,w),(v,w)\rangle+o(|(v,w)|^3)\quad\text{as }|(v,w)|\to 0,
$$
where the error is uniform in $x,\,y$ (affine, to be precise). We write $Q_1(v,w)=\langle F''(x,y)(v,w),(v,w)\rangle$ and $Q_2(v,w)=\langle G''(x,y)(v,w),(v,w)\rangle=2\langle M,xv^T+wy^T\rangle^2$. Therefore, the zeroes of $Q_1$ are also zeroes of $Q_2$, so we can apply Lemma \ref{lemma:difference_two_nonnegative_quadratics} to obtain that $Q_1\geq \alpha Q_2$, where $\alpha>0$. This choice requires a bit of justification, since both $Q_i$ depend on $x,\,y$ of length one. By examining the proof of Lemma \ref{lemma:difference_two_nonnegative_quadratics}, we see that we can take $\alpha$ to be proportional with the minimal eigenvalue of $F''(x,y)$. It is elementary to check that the Hessian of a nontrivial biquadratic in $(x,y)$ cannot be zero at any point with $|x|=|y|=1$, so the minimal eigenvalue must be strictly positive.

Therefore, around the zeroes of $F$ with $|x|=|y|=1$, we have, with uniform error in $(x,y)$ as $|(v,w)|\to 0$,
\begin{align*}
    F(x+v,y+w)&=\tfrac{1}{2}\langle F''(x,y)(v,w),(v,w)\rangle+o(|(v,w)|^3)\\
    &\geq \tfrac{\alpha}{2}\langle G''(x,y)(v,w),(v,w)\rangle+o(|(v,w)|^3)\\
    &= \alpha G(x+v,y+w),
\end{align*}
so there exists a small relative neighborhood $\mathcal N\subset \{(a,b)\colon |a|=|b|=1\}$ of the zero set of $F$ where $F\geq \alpha G$. Outside of this set, $F>0$, so we can find a constant $\beta>0$ such $F\geq \beta >0$ in $\{(a,b)\colon |a|=|b|=1\}\setminus \mathcal N $. We conclude that in this set we have
$$
F\geq \frac{\beta}{\sup_{|a|=|b|=1}G(a,b)} G,
$$
so, by possibly making $\alpha$ smaller, we can conclude that $F\geq \alpha G$ everywhere. Therefore, $F$ cannot be a weak extremal, which concludes the proof of the second implication.
\end{proof}

It is clear that weak extremals are at the boundary of $\pbnm$. We have the following geometric characterization:
\begin{prop}\label{prop:WE_relint_of_face}
Let $F$ be a nonnegative biquadratic, $F\neq 0$. Let $K$ be the minimal face $K$ of $\pbnm$ that contains $F$ in its relative interior. The following are equivalent:
\begin{enumerate}
    \item $F$ is a weak extremal that is not strong extremal;
    \item $K$ has dimension at least two and contains no squares.
\end{enumerate}
\end{prop}

\begin{proof}
By definition, $F$ is a strong extremal if and only if $K$ has dimension one. Assume henceforth that $K$ has dimension at least two and let $r>0$ be the radius of a ball around $F$ such that $B(F,r)\cap K$ is contained in the relative interior of $K$. If $K$ contains a square $S$, we have that $F-rS\in K$. In particular this means $F-rS\geq 0$ so that $F$ cannot be weak extremal. Conversely, suppose that $F$ is not weak extremal so that we can write $F=G+S$, where $G\in\pbnm$ and $S$ is a square. If $S\notin K$, a simple geometric argument shows that $G=F-S\notin\pbnm$, so we must have that $S\in K$, which concludes the proof. 
\end{proof}

\section{Weak extremals which are strong extremals}\label{sec:weak-strong}

With this new characterization of weak extremals, we aim to prove Theorem \ref{thm:main}.  In the particular case when the biquadratic $F \in \mathrm{PB}_3$ has nine \emph{distinct} zeros, the proof is significantly simpler then in the general case, but illustrative of the overarching strategy.  Thus we begin the section by proving this specific case.

We write $\Pi_1(z) = x$ for the projection onto the first vector of $z = (x,y) \in \ProjP^2 \times \ProjP^2$. We make use of the following results from real algebraic geometry.

\begin{lemma}[\cite{Quarez2015}]\label{lem:quarez_same_component}
Let $F \in \mathrm{PB}_3$ be a nonnegative biquadratic with two different zeros $A$ and $B$ in $\ProjP^2 \times \ProjP^2$ such that $\Pi_1(A) = \Pi_1(B)$.  Then $F$ is a sum of squares and has infinitely many zeros.
\end{lemma}

\begin{theorem}[\cite{KS2018}]\label{theorem:ternary_sextic_nine_zeros}
Let $f\in \R[x,y,z]$ be a nonnegative ternary sextic that is not a sum of squares.
Suppose that $\mathcal V(f) = \{x\in \R\P^2 \colon f(x) = 0\}$ is a set of 9 real points. Suppose that the Hessian of $f$ has rank $2$ at each of these 9 points.
Then there is a 2--dimensional linear space of sextics passing through $\mathcal V(f)$ which contains a unique (up to positive scaling) extreme nonnegative ternary sextic $g$ that is not a sum of squares.
Moreover, this unique sextic $g$ has a tenth zero which generically is not in $\mathcal V(f)$, but may degenerate into a higher (namely $A_3$-) singularity contained in $\mathcal V(f)$ (meaning that the Hessian of $g$ is only rank $1$ at one of the 9 points). 
Finally, every other nonnegative ternary sextic vanishing on $\mathcal V(f)$ which has finitely many real zeros and a tenth real zero is a positive rescaling of $g$.
The face of the cone of nonnegative ternary sextics that vanish at the $9$ real zeros of $f$ is a $2$-dimensional cone spanned by the square of a cubic and the extreme form $g$.
\end{theorem}

\begin{theorem}\label{thm:WE_nine_zeros}
Let $F \in \mathrm{PB}_3$ be weak extremal and suppose $F$ has nine distinct real zeros.  Then $F$ is strong extremal.
\end{theorem}
\begin{proof}
We proceed with a proof by contradiction.
If $F$ is not strong extremal, by Proposition \ref{prop:WE_relint_of_face} it is in the relative interior of a face $K$ of dimension at least two that does not contain squares.  Let $H_1$ and $H_2$ be two distinct biquadratics on the boundary of $K$ such that $F = H_1 + H_2$.
Then $H_1$ and $H_2$ share the same 9 real zeros as $F$.  Moreover, $H_1$ and $H_2$ must each have an additional zero, either distinct from the 9 zeros of $F$ or among one of the 9 zeros, now with a higher singularity.

Suppose first that both $H_1$ and $H_2$ have a tenth real zero not in $\V(F)$.  Call them $z_1 = (x_1,y_1)$ and $z_2 = (x_2,y_2)$ for $H_1$ and $H_2$ respectively, with $z_1 \neq z_2$.  By Lemma \ref{lem:quarez_same_component}, $x_i$ must be distinct from the $x$-component of any of the original nine zeros, else $H_i$ will be a sum of squares, contradicting the fact that $F$ is weak extremal.  Likewise for $y_1$ and $y_2$.  If $x_1 = x_2$, then $y_1 \neq y_2$ and we may argue using the $y$-matrix instead.  Thus, assume that $x_1 \neq x_2$.

Letting $H_i(x) = y^TT_i(x)y$, for $i = 1,2$, define $h_i = \det T_i(x)$.  
Both $h_1,h_2$ are ternary sextics and they share the nine zeros in $\ProjP^2$ coming from projecting the zeros of $F$ onto the first vector.  Moreover, they each have a tenth zero distinct from each other and distinct from the original nine. They both have only finitely many real zeros because every real zero of $h_i$ lifts to a zero of $H_i$.
Thus by Theorem \ref{theorem:ternary_sextic_nine_zeros}, both  $h_1$ and $h_2$ are extreme ternary sextics that vanish on the same nine zeros with distinct tenth zeros.
However, this contradicts the uniqueness of the extreme ternary sextic vanishing on the 9 zeros.

Suppose that $H_1$ has a new tenth real zero while $\V(H_2) = \V(F)$.
Again $h_1$ and $h_2$ are both extreme ternary sextic by Theorem \ref{theorem:ternary_sextic_nine_zeros} vanishing on the 9 zeros which again contradicts uniqueness.  The same argument follows in the third case where $H_1$ and $H_2$ both have a tenth zero in $\V(F)$, but not at the same zero exhausting all possibilities.   
\end{proof}

The proof of Theorem \ref{thm:main} follows similarly to Theorem \ref{thm:WE_nine_zeros}.  However, it relies on a generalization of the Kunert-Scheiderer result proven in the next section.  For the moment, we will assume all results from Section \ref{sec:ternary_sextics} to be true. We also need the following result:

\begin{lemma}\label{lem:roots}
Let $F \in \mathrm{B}_3$ be a biquadratic and let $T(x)$ be the $x$-matrix of $F$.  Suppose $(v,w) \in \ker(F''(x_0,y_0))$ for a zero $(x_0,y_0)$ of $F$ with $F'(x_0,y_0) = 0$. Let $f = \det T(x)$, then $v \in \ker f''(x_0)$.

If furthermore $F \in \mathrm{PB}_3$ is a nonnegative real biquadratic that is not a sum of squares and $(x_0,y_0)$ is a real zero of $F$, then the following statements hold.
\begin{enumerate}
    \item If there exist a pair of linearly independent points, $(v_1,w_1)$ and $(v_2,w_2)$ in the kernel of $F''(x_0,y_0)$, where $v_j$ and $w_j$ are nonzero, then $v_1 \neq v_2$ and $w_1 \neq w_2$.
    \item The dimension of the kernel of $F''(x_0,y_0)$ cannot be greater than 4.
\end{enumerate}
\end{lemma}
\begin{proof}
After a change of coordinates, we can assume that $(x_0,y_0) = (1,0,0,1,0,0)$ and $v = (0,1,0)$.  
We then solve the equations $F(x_0,y_0) = 0$, $F'(x_0,y_0) = 0$, and $F''(x_0,y_0)(v,w) = 0$ using a computer algebra system in order to compute the Hessian of $\det T(x)$ at $(1,0,0)$. This shows that only the $(3,3)$-entry of this Hessian can be nonzero, i.e.~it takes the form
$$\begin{bmatrix}
0 & 0 & 0 \\
0 & 0 & 0 \\
0 & 0 & \ast
\end{bmatrix}$$
and so clearly contains $(0,1,0)$ in its kernel.
If we further assume that $(0,0,1,w')$ is in the kernel of $(F''(x_0,y_0))$, we get that the Hessian of $\det T(x)$ is the zero matrix.

So now we assume that $F$ is nonnegative and that $(x_0,y_0)$ is a real zero and show that if $(v_1,w_1), (v_2,w_2) \in \ker(F''(x_0,y_0))$, then $v_1 \neq v_2$.  
Suppose that $v_1 = v_2 = v$.  
Then the following vectors are contained in the kernel,
$$\begin{pmatrix}  x_0  \\ 0 \end{pmatrix},
\begin{pmatrix} 0 \\ y_0 \end{pmatrix},
\begin{pmatrix} v \\ w_1 \end{pmatrix},
\begin{pmatrix} v \\ w_2 \end{pmatrix}.$$
Consider the $3 \times 3$ block of $F''(x_0,y_0)$ corresponding to the partials with respect to the $y$ variables, denoted by $F''_y$.  
Clearly it has $y_0$ in its kernel and it also has $w_1-w_2$ in its kernel, since the vector $(v,w_1)-(v,w_2)$ is in the kernel of $F''(x_0,y_0)$.
Hence the rank of $F''_y$ is at most one.
Let $y_0$ and $v$ be two linearly independent vectors in $\ker F''_y$.
This implies that $F(x_0,y_0+sv) = 0$ for any $s \in \R$.  Thus $F$ has infinitely many zeros and must be a sum of squares, contradicting our assumptions. 
Symmetrically, the same argument also works for $w_1$ and $w_2$.

For the last statement, if the dimension of $\ker(F''(x_0,y_0))$ is 5, we again find that the $3 \times 3$ block $F''_y$ is rank one and thus we get that $F$ is a sum of squares, a contradiction. 
\end{proof}

With this result and our generalizations of Kunert and Scheiderer's work in the following section, we can now proof the main result of our paper.

\begin{proof}[Proof of Theorem~\ref{thm:main}]
Only one implication needs proof. Suppose that $F\in\mathrm{PB}_3$ is a weak extremal, but not strong extremal. By Theorem~\ref{thm:char_we}, we have that the dimension of $L_F$ is 9.  
If $F$ is not strong extremal, then it can be written as the sum $F = H_1 + H_2$ of two nonnegative biquadratics which are linearly independent, nontrivial, on the boundary of the face of $\mathrm{PB}_3$ contaning $F$ in its relative interior, and not squares by Proposition \ref{prop:WE_relint_of_face}.   
As before, $H_1$ and $H_2$ must have picked up a new zero, either distinct from $\V(F)$ or in such a way that the kernel of the Hessian has increased in dimension.

Write $H_i(x,y) = y^TT_i(x)y$ and let $h_i = \det T_i(x)$, for $i = 1,2$. 
By Lemma \ref{lem:roots} and Theorem \ref{thm:exteme_ternary_sextic_conditions}, $h_1$ and $h_2$ are extreme ternary sextics, not a sum of squares, and in the same pencil.  By the proof of Theorem \ref{thm:exteme_ternary_sextic_conditions}, such an extreme ternary sextic must be unique, giving a contradiction.
\end{proof}

\section{Extremal nonnegative ternary sextics}\label{sec:ternary_sextics}
We would like to prove the following generalization of the Kunert-Scheiderer result.
\begin{theorem}\label{thm:KS_generalization}
A nonnegative ternary sextic $f\in \R[x,y,z]$ that is not a sum of squares is (strong) extremal in the convex cone of nonnegative ternary sextics if and only if the curve 
\[
\V(f) = \{x\in \C\P^2\colon f(x) = 0\} \subset \C\P^2
\]
is a rational sextic and all of its singularities are real. 
\end{theorem}

We will do this in several steps along similar lines as Kunert and Scheiderer in their paper \cite{KS2018}.
\begin{lemma}\label{lem:ternary6SOS}
Let $f\in\R[x,y,z]$ be a nonnegative ternary sextic. 
\begin{itemize}
    \item If $f$ is reducible in $\R[x,y,z]$, then $f$ is a sum of squares.
    \item If $f$ is reducible in $\C[x,y,z]$, then $f$ is a sum of squares.
    \item If $f$ has infinitely many real zeros (that is zeros in $\R\P^2$), then $f$ is a sum of squares.
\end{itemize}
\end{lemma}

\begin{proof}
\cite[Lem. 3.1 and Prop. 3.2]{clrzbMATH03645183}
\end{proof}

\begin{lemma}\label{lem:HessianSOS}
Let $f$ be a nonnegative ternary sextic and suppose that there is a point $P = (x:y:z)\in \R\P^2$ such that the Hessian of $f$ at $P$ is the zero matrix. Then $f$ is a sum of squares. 
\end{lemma}

\begin{proof}
We change coordinates and assume that $P = (0:0:1)$. If the Hessian matrix of $f$ at $P$ is the zero matrix, then $f$ is in fact a biform: it has degree $2$ in $z$ so that we can homogenize it in $z$ with a new variable $w$ and write it as
\[ f = 
\begin{pmatrix} z & w\end{pmatrix} 
\begin{pmatrix}
a_4 (x,y) & a_5(x,y) \\
a_5 (x,y) & a_6(x,y)
\end{pmatrix} 
\begin{pmatrix}
z \\ w
\end{pmatrix}
\]
where $a_i\in\R[x,y]$ is homogeneous of degree $i$. If such a biform is nonnegative, it is a sum of squares by \cite[Thm. 1.1]{GGMzbMATH06572970}. This follows by considering its Newton polytope. For a more detailed discussion of this case, see also \cite[Sec. 3]{zbMATH07102076}.
\end{proof}

The curve $\V(f) = \{x\in \C\P^2\colon f(x) = 0\}\subset \C\P^2$ defined by a nonnegative ternary sextic is singular at every real zero of $f$ in $\R\P^2$. 
If such a zero is an isolated double point and $f$ is not a sum of squares, there are only three possibilities, since the degree of $f$ is $6$ and the Hessian has rank at least $1$ at each of them by the previous Lemma \ref{lem:HessianSOS}. It is either an $A_1$-, $A_3$-, or $A_5$-singularity. Locally, around an $A_n$-singularity, such a curve is given by the equation $x^2 + y^{n+1} = 0$ (up to a real change of coordinates) because the polynomial is globally nonnegative. 

At each singularity, we are interested in its \emph{$\delta$-invariant} or \emph{order}, see \cite[Sec. 3.11]{CasasAlverozbMATH01497487}. The $\delta$-invariant of an $A_n$ singularity is the smallest integer greater than or equal to $n/2$, as \cite[Thm. 3.11.12]{CasasAlverozbMATH01497487} shows by the resolution of $A_n$-singularities discussed in \cite[Sec. 4.3]{Hartshorne}.

A \emph{pencil} is a $2$-dimensional linear space (which is the same as a $1$-dimensional projective space).
\begin{lemma}\label{lem:resolutionssing}
Let $f\in \C[x,y,z]$ be an irreducible ternary sextic. Suppose that the plane curve $\V(f)\subset \C\P^2$ is rational and that it has only singularities of type $A_1$, $A_3$, or $A_5$.
If it has a node $P$ (which is an $A_1$-singularity), then the set of ternary sextics that have the same singularities as $\V(f)$ outside of $P$ is a pencil of sextics containing the square of a cubic.
If it has only singularities of type $A_3$ or $A_5$, then it has a singularity $P$ of type $A_3$. In this case, the set of ternary sextics that have the same singularities as $\V(f)$ outside of $P$ and at $P$ a singularity of type $A_1$ is again a pencil of sextics containing the square of a cubic.
\end{lemma}

\begin{proof}
Since the curve $\V(f)$ is rational, the sum of the $\delta$-invariants of all its singularities is the genus of a smooth sextic in $\C\P^2$, which is $10$. If the curve has no singularities of type $A_1$, it must have $2$ of type $A_3$ since an $A_5$-singularity has $\delta$-invariant $3$ because the only partition of $10$ with parts restricted to $2$ and $3$ is $10 = 3+3+2+2$.

We show the other claims by resolving the singularities of $C = \V(f)$ except for $P$ and then considering the linear system of sextics as in the claim on the blown up surface to see that they form a pencil. 

We first discuss the resolutions of singularities for the three types of double points that a nonnegative ternary sextic can have. This is also discussed in \cite[Sec. V.3]{Hartshorne}.

The embedded resolution of a node $Q$ (an $A_1$-singularity) is given by the blow up $\pi\colon X\ratto \C\P^2$ of $\C\P^2$ in that point $Q$. The Picard group of this blow up is $\Z^2$ and is generated by the strict transform $\wh{L}$ of a line that does not contain $Q$ and the exceptional divisor. The intersection product is given by $\wh{L}.\wh{L} = 1$, $\wh{L}.E = 0$, and $E.E = -1$. The strict transform of $C$ is $\pi^*(6L) - 2E = 6\wh{L} - 2E$. The self-intersection of the strict transform of $C$ is therefore $32 = 6^2 + 4 \cdot(-1)$.

The embedded resolution of an $A_3$-singularity $Q$ is obtained in two steps. First, we blow up $Q$ to obtain $X_1\ratto \C\P^2$. Then the strict transform of $C = \V(f)$ still has an $A_1$-singularity on the exceptional divisor $E_1\subset X_1$. To resolve this, we blow up this singularity to obtain $X_2 \ratto X_1$ in which the strict transform of $C$ is smooth. Now the Picard group of $X_2$ is $\Z^3$. As generators, we pick the strict transform $\wh{L}$ of a line in $\C\P^2$ that does not contain $Q$, the strict transform $\wh{E_1}$ of the exceptional divisor $E_1\subset X_1$ with respect to the blow up $X_2\ratto X_1$, and the exceptional divisor $E_2$. Now the intersection product is given by the matrix
\[ 
S_3 = \begin{pmatrix}
1 & 0 & 0 \\
0 & -2 & 1 \\
0 & 1 & -1
\end{pmatrix}\]
with our chosen basis. This is because $\wh{E_1}$ is the same as $\pi_2^*(E_1) - E_2$. That class has self-intersection $-2$ and it intersects $E_2$ in one point.
The class of the strict transform of $C$ is $6 \wh{L} - 2\wh{E_1} - 4 E_2$. So its self-intersection is $28 = (6,-2,-4) S_3 (6,-2,-4)^t$.

For the embedded resolution of an $A_5$-singularity $Q\in \C\P^2$, we have to blow up three times. First, we blow up $\C\P^2$ in $Q$ again and write this as $\pi_1\colon X_1 \ratto \C\P^2$. Now, the strict transform of $C$ has an $A_3$-singularity on the exceptional divisor that we blow up next to get $\pi_2\colon X_2 \ratto X_1$. The strict transform of $C$ on $X_2$ still has a singularity on the exceptional divisor, which is a node ($A_1$-singularity) this time. A last blow up $\pi_3 \colon X_3 \to X_2$ at this point finally gives a smooth strict transform of $C$. The Picard group of $X_3$ is $\Z^4$. As a basis, we choose the strict transform $\wh{L}$ of a line in $\C\P^2$ not containing $Q$, the strict transform $\wh{E_1}$ of the exceptional divisor $E_1\subset X_1$, the strict transform $\wh{E_2}$ of the exceptional divisor $E_2\subset X_2$, and the exceptional divisor $E_3$. The Picard group of $X_2$ the same as in the previous case. The center of the last blow up $\pi_3$ lies on $E_2$ but it is not the intersection point of $E_2$ and the strict transform of $E_1$ in $X_2$ (see \cite[Cor. V.3.7 and Exm. V.3.9.5]{Hartshorne}). So the intersection matrix of $\pic(X_3)$ restricted to the sublattice generated by $\wh{L}$, $\wh{E_1}$ and $\wh{E_2}$ is the same as before. For the same reason as before, we have $\pi_3^*(E_2) = \wh{E_2} + E_3$ so that in total, the intersection matrix in this case is
\[
S_4 = \begin{pmatrix}
1 & 0 & 0 & 0 \\
0 & -2 & 1 & 0 \\
0 & 1 & -2 & 1 \\
0 & 0 & 1 & -1
\end{pmatrix}.
\]
The class of the strict transform of $C$ in $\pic(X_3)$ is $6\wh{L} - 2\wh{E_1} - 4\wh{E_2} - 6 E_3$. Its self-intersection is $24 = (6,-2,-4,-6) S_4 (6,-2,-4,-6)^T$ this time.

In the case that $C = \V(f)$ has a node (which has $\delta$-invariant $1$), the sum of the $\delta$-invariants of the singularities that we resolve is $9$. For instance, this could be $9$ nodes ($A_1$-singularities) or $7$ nodes and one $A_5$-singularity and so on. Either way, the self-intersection of the class of the strict transform of $C$ in the resolution of singularities is $0$. This follows from the above cases: In every case, the self-intersection of the strict transform drops by $4\delta$ for every singularity. Since the $\delta$-invariants sum to $9$, the self-intersection number drops from $36$ by $4*9 = 36$.

The same argument applies if $C$ has no node and we pick a singularity of type $A_3$ that we only partially resolve by one blow-up. The strict transforms of sextics that we consider in the claim under such a blow-up is also $6\wh{L} - 2E$, the same as in the case of a node. Therefore, the self-intersection of the class of sextics that have the same singularities as $\V(f)$ outside of the $A_3$-singularity $P$ and a node at $P$ still has self intersection $0$ on the appropriate embedded resolution of singularities.

In either case, the linear system on the blow-up contains the strict transform of the square of a cubic. The cubic is given by the following $9$ linear conditions: For every $A_1$-singularity of $\V(f)$ (except, possibly, for $P$), we require the cubic to vanish there. For every $A_3$-singularity (except, possibly, for $P$), we require that the cubic vanishes and that the gradient is orthogonal to the kernel of the Hessian of $f$ at that point. This gives (at most) two linear conditions. At an $A_5$ singularity, we require that the point be an inflection point of the cubic with a gradient that is orthogonal to the kernel of the Hessian. This gives (at most) $3$ linear conditions on the cubic. If $P$ is an $A_3$-singularity, then the requirement for the cubic is to only vanish at $P$ (because we want the square to have a double point at $P$ then). Together, this gives at most $9$ linear conditions so that there is always the square of a cubic in the linear system that we consider.

On the partial resolution of singularities $X$ obtained by $9$ successive blow-ups, the strict transforms of the square of the cubic and of the irreducible sextic do not intersect since they cannot have a component in common and the self-intersection is $0$ by the generalization of B\'ezout's Theorem to surfaces, see \cite[Prop. V.1.4]{Hartshorne}. This linear system is therefore base-point free.

So we consider the divisor class of the strict transform of $C$ in the partial resolution of singularities $X$ obtained by $9$ successive blow-ups. The associated line bundle on $X$ has at least two global sections, namely one corresponding to the strict transform of $C$ and the other one corresponding to the strict transform of the cubic squared. 
We consider the map associated to this linear system. Since the self-intersection of the divisors in this system is $0$, the map cannot be generically finite. So the image has dimension $1$. The degree of this curve is a lower bound for the number of irreducible components of a general divisor in the linear system $|\wh{C}|$ and must therefore be $1$ because $\wh{C}$ is irreducible. So the linear system of such sextics is a pencil as claimed.
\end{proof}

\begin{theorem}\label{thm:exteme_ternary_sextic_conditions}
Let $f\in\R[x,y,z]_6$ be a nonnegative ternary sextic that is not a sum of squares. Then $f$ is an extreme ray of the cone of nonnegative polynomials if and only if the following conditions hold.
\begin{enumerate}
    \item $f$ is irreducible in $\C[x,y,z]$.
    \item All singularities of $\V(f)\subset \C\P^2$ are real.
    \item The curve $\V(f)\subset \C\P^2$ is rational, i.e.~the sum of the $\delta$-invariants over all singularities of $\V(f)$ is $10$.
\end{enumerate}
\end{theorem}

\begin{proof}
First, let $f$ be an extreme ray of the cone of nonnegative ternary sextics. Lemma \ref{lem:ternary6SOS} implies that $f$ is irreducible in $\C[x,y,z]$ and that the set of real zeros of $f$ is finite. Since $f$ is extremal, it must be uniquely determined by its real singularities. Therefore, they must all (in the sense of claim 3) be real. Indeed, if the form had fewer real zeros, there would be another ternary sextic $g$ with the same real zeros and by perturbation $f \pm \epsilon g$ would still be nonnegative. 

Conversely, suppose that $f\in\C[x,y.z]$ is irreducible, $\V(f)\subset \C\P^2$ is rational and all singularities of $\V(f)$ are real. Consider the pencil $L$ of sextics as in Lemma \ref{lem:resolutionssing} with a choice of point $P$. The form $f$ is an element of this pencil and so is the square $c^2$ of a real cubic $c\in\R[x,y,z]$ (determined as in the proof of Lemma \ref{lem:resolutionssing} by the linear conditions that $c$ vanishes at every $A_1$-singularity of $\V(f)$ except for $P$, vanishes at every $A_3$-singularity of $\V(f)$ with a determined tangent direction given by the kernel of the Hessian of $f$ at that point, and an inflection point at every $A_5$-singularity of $\V(f)$ with determined tangent direction). This shows that the face of the cone of nonnegative ternary sextics contained in the $2$-dimensional space $L$ is a cone spanned by $c^2$ and $f$. Therefore, $f$ spans an extreme ray. Indeed, $f$ is extremal because any sufficiently small perturbation of $f$ in direction $-c^2$ will make it negative for some points close enough to $P$.
\end{proof}

\section{Counterexample}\label{sec:counterexample}
In this section we prove Theorem \ref{thm:main-ext_not_ext_det}, which we recast as follows:

\begin{conjecture}[\cite{HarutyunyanMiltonCPAM} Conjecture 2.8, \cite{harutyunyan2021extreme} Conjecture 2.7]\label{conj:counterexample}
Let $F\in \mathrm{PB}_{3}$ be a nonnegative biquadratic.  If $F$ is a weak extremal, then the determinant of the acoustic tensor, $\det T(x)$, is an extremal polynomial different from a perfect square.
\end{conjecture}
Here we will construct a counterexample to Conjecture \ref{conj:counterexample} coming from \cite{BS2020}, which is even strong extremal (and in particular weak extremal).
Let $f_{p,q}(x,y)$ be a nonnegative biquadratic with the following 9 zeros:
    $$S = \{(1, 1, 1; 1, 1, 1),(1, 1, -1; 1, 1, -1),(1, -1, 1; 1, -1, 1),(-1, 1, 1; -1, 1, 1),(1, p, 0; q, 1, 0),$$
$$    (1, -p, 0; -q, 1, 0),(0, 1, q; 0, p, 1), (0, 1,-q; 0,-p, 1), (0, 0, 1; 1, 0, 0)\}.$$

\begin{theorem}[\cite{BS2020} Theorem 7]\label{theorem:nine_zeros_extremal}
For $(p,q) \in [0,\frac{1}{\sqrt{2}}]\times [0,\sqrt{2}]$ such that $2p>q$ and $(p^2-1)^2q^2 \geq p^2$,
$f_{p,q}(x,y)$ is not a square, has exactly these nine zeros, and is a strong extremal.
\end{theorem}

From this family of strong extremals, we fix $p=1/2$ and $q=3/4$. This satisfies Theorem \ref{theorem:nine_zeros_extremal} and is thus a strong extremal different from a square.  The associated $T(x)$ matrix is (a postive scaling of)

$$T(x) = \left(\begin{array}{rrr}
100 \, x_{0}^{2} + 192 \, x_{1}^{2} & -222 \, x_{0} x_{1} & -70 \, x_{0} x_{2} \\
-222 \, x_{0} x_{1} & 27 \, x_{0}^{2} + 225 \, x_{1}^{2} + 192 \, x_{2}^{2} & -222 \, x_{1} x_{2} \\
-70 \, x_{0} x_{2} & -222 \, x_{1} x_{2} & 165 \, x_{0}^{2} + 27 \, x_{1}^{2} + 100 \, x_{2}^{2}
\end{array}\right).$$

It turns out that $\det T(x)\in\R[x_0,x_1,x_2]$ is not extremal providing a counterexample to Conjecture \ref{conj:counterexample}.  This follows from work by Kunert and Scheiderer summarized in Theorem \ref{theorem:ternary_sextic_nine_zeros}.
Indeed, $\det T(x)$ is a ternary sextic which vanishes on exactly nine zeros by Theorem \ref{theorem:nine_zeros_extremal}.
However, at each of the zeros, its $3\times 3$ Hessian has (maximal) rank 2 -- it cannot be less because of homogeneity and the vanishing of the gradient by Euler's identity. 
Since $\det T(x)$ has no tenth root, it cannot be extremal by Theorem \ref{theorem:ternary_sextic_nine_zeros}.

\appendix

\section{Bounds for effective properties of composites}\label{sec:composites}

First, we briefly explain what are the effective properties of composite materials and what kind of bounds  we are interested in. Second, we outline the translation method for obtaining such bounds. Then we show that translations by quasiconvex quadratic forms give good bounds from below. Lastly, we show that, among quasiconvex translations, the weak extremal ones give the best bounds.

We largely follow the presentation from Milton's monograph \cite{MiltonBook}, see Chapters 1, 12, 13, 21, 24, and 25 therein. We will consider a two-phase (elastic) energy of the form 
 $$
    E[u]=\int_{\Omega} \langle \nabla u(x),\CC(x)\nabla u(x)\rangle \dif x,\quad u\colon\Omega\rightarrow\R^n
    $$
    where $\Omega\subset\R^n$, $n=2,3$, represents an elastic body, say, a ball, and
    $\CC(x)=\CC_1\mathbf{1}_{\Omega_1}(x)+\CC_2\mathbf{1}_{\Omega_2}(x)$. Here $\CC_i$ are 4-tensors that induce positive definite quadratic forms on $\R^{n\times n}$ and $\{\Omega_i\}_{i=1,2}$ partition $\Omega$. Then $\CC(x)$ represents the microscopic \emph{stiffness tensor} that links the strain and stress of the deformation. 
    We are interested in the macroscopic description of the stiffness matrix, given by
    \begin{equation}\label{eq:eff}
      \langle \xi,\CC_*\xi \rangle\coloneqq \inf_{\varphi|_{\partial\Omega}=0} \int_{\Omega} \langle \xi+\nabla\varphi(x),\CC(x) (\xi+\nabla \varphi(x)) \rangle\dif x \quad \text{for }\xi\in\R^{n\times n},
    \end{equation}
    where the infimum is taken over $\varphi\in C_c^\infty(\Omega,\R^n)$. We will refer to $\CC_*$ as the \emph{effective} stiffness tensor of $\CC$.

While the computation of $\CC_*$ is, in general, very difficult and heavily depends on the shape of the two phases $\Omega_i$, finding optimal bounds for this quantity is very helpful for practical purposes \cite[Ch. 21]{MiltonBook}. For instance, a very simple bound is obtained by simply taking $\varphi=0$ in \eqref{eq:eff}, leading to the inequality
\begin{equation}\label{eq:AM}
\CC_*\leq \int_{\Omega} \CC(x)\dif x,
\end{equation}
which is understood in the sense of quadratic forms on $\R^{n\times n}$. A slightly more involved calculation gives the harmonic mean bound from below
\begin{equation}\label{eq:HM}
    \CC_*\geq \left(\frac{1}{|\Omega|}\int_{\Omega}\CC(x)^{-1}\dif x\right)^{-1}=\left(\frac{|\Omega_1|}{|\Omega|}\CC_1^{-1}+\frac{|\Omega_2|}{|\Omega|}\CC_2^{-1}\right)^{-1}.
\end{equation}
This can be deduced from the arithmetic mean bound \eqref{eq:AM} by an elementary duality argument, see \cite[Ch. 13.1]{MiltonBook} or \cite{MiltonCPAM}. In fact, the harmonic mean bound can be traced back to work done in Elasticity Theory in the early 20th century \cite{Reuss,Voigt}.

If the exact computation of $\CC_*$ depends on the geometry of $\Omega_i$, the bounds in \eqref{eq:AM} and \eqref{eq:HM} depend only on their \emph{volume fractions}. This motivates the practical question of optimality of these bounds: do there exist shapes of $\Omega_i$ with fixed volume fraction such that the bounds are attained? For an explicit example see, for instance \cite[Ch. 22.4]{MiltonBook}. In many examples, optimal bounds give certain correlations between various physical (microscopic and/or macroscopic) quantities that arise in practical problems.

To improve the basic bound in \eqref{eq:HM}, we will sketch some considerations related to the \emph{translation method}, introduced independently by Murat and Tartar \cite{MuratTartar,tartar1979,tartar1985} and Lurie--Cherkaev \cite{lurie1984exact}. By a remarkable observation of Milton \cite{MiltonCPAM}, there is a correspondence between this method and the variational principles of Hashin--Shtrikman \cite{hashin1962variational,hashin1963variational}, which were used to derive optimal bounds in conductivity and linear elasticity that have long been the benchmark against which many experiments are compared, see \cite[Ch. 23]{MiltonBook}. 

We now briefly describe how the translation method can be used in our context. The basic idea is to consider a constant tensor $\TT$, the \emph{translation}, which induces a quadratic form on $\R^{n\times n}$, and to compare $(\CC(x)-\TT)_*$ with $\CC_*-\TT$. The latter is easily expressed by \eqref{eq:eff} for $\xi\in\R^{n\times n}$ by 
$$
\langle \xi,(\CC_*-\TT)\xi \rangle= \inf_{\varphi|_{\partial\Omega}=0} \int_{\Omega} \langle \xi+\nabla\varphi(x),\CC(x) (\xi+\nabla \varphi(x)) \rangle\dif x-\langle \xi,\TT\xi \rangle.
$$
Now we make the crucial observation that, if the translation we use is \emph{quasiconvex}, we have that 
\begin{align*}
\langle \xi,(\CC_*-\TT)\xi \rangle&\geq \inf_{\varphi|_{\partial\Omega}=0} \int_{\Omega} \langle \xi+\nabla\varphi(x),\CC(x) (\xi+\nabla \varphi(x)) \rangle-\langle \xi+\nabla\varphi(x),\TT (\xi+\nabla \varphi(x)) \rangle\dif x\\
&= \inf_{\varphi|_{\partial\Omega}=0} \int_{\Omega} \langle \xi+\nabla\varphi(x),(\CC(x)-\TT) (\xi+\nabla \varphi(x)) \rangle\dif x\\
&=\langle \xi,(\CC-\TT)_*\xi \rangle.
\end{align*}
Since the inequality holds for all $\xi\in\R^{n\times n}$, we can infer that
\begin{equation}\label{eq:translation_method}
\CC_*\geq \TT+(\CC-\TT)_*.
\end{equation}
Of course, determining the effective stiffness tensor $\CC-\TT$ is no easier than determining the effective stiffness tensor $\CC$ , but here we can use the basic bound \eqref{eq:HM} to obtain a computable formula
\begin{equation}\label{eq:translation_bound}
    \CC_*\geq \TT+\left(\frac{1}{|\Omega|}\int_{\Omega}(\CC(x)-\TT)^{-1}\dif x\right)^{-1},
\end{equation}
which can then be studied for tensors $\TT$ inducing quasiconvex quadratic forms such that the inverse under the integral is well defined, i.e.,
\begin{equation}\label{eq:restriction}
\CC(x)-\TT\geq 0\quad\text{for all }x\in\Omega.
\end{equation}
The difficult question then becomes choosing such tensors $\TT$ for which there exist geometries $\{\Omega_i\}_{i=1,2}$ such that the bounds are attained. This is well beyond the scope of this outline.
 
We have thus showed how quasiconvex quadratic forms (hence \textbf{positive biquadratics}, see Section \ref{sec:prel}) enter the computation of important optimal bounds in the Theory of Composites. 
To conclude this presentation, we will show that \textbf{weak extremal} quasiconvex quadratic forms (see Definition \ref{def:WE}) are the right building blocks to investigate the bounds \eqref{eq:translation_method}. 

To this end, we will show that translation by a nonnegative semidefinite quadratic form $\SSS$ can only make the right hand side of \eqref{eq:translation_method} smaller, hence leads to a worse bound from below. To see that this is the case, we write $\tilde{\CC}=\CC-\TT$ and derive
\begin{align*}
\langle \xi,(\tilde\CC-\SSS)_*\xi \rangle&= \inf_{\varphi|_{\partial\Omega}=0} \int_{\Omega} \langle \xi+\nabla\varphi(x),\tilde\CC(x) (\xi+\nabla \varphi(x)) \rangle-\langle \xi+\nabla\varphi(x),\SSS (\xi+\nabla \varphi(x)) \rangle\dif x\\
&\leq \inf_{\varphi|_{\partial\Omega}=0} \int_{\Omega} \langle \xi+\nabla\varphi(x),\tilde\CC(x) (\xi+\nabla \varphi(x)) \rangle\dif x-\langle \xi,\SSS\xi \rangle\\
&=\langle \xi,(\tilde\CC_*-\SSS)\xi \rangle,
\end{align*}
where we used Jensen's inequality for the convex quadratic form induced by $\SSS$, see Lemma \ref{lem:convex}. Therefore
\begin{align*}
    (\CC-\TT)_*-\SSS\geq (\CC-\TT-\SSS)_*,
\end{align*}
which implies that
\begin{equation*}
    \TT+(\CC-\TT)_*\geq\TT+\SSS+(\CC-\TT-\SSS)_*,
\end{equation*}
so the right hand side of \eqref{eq:translation_method} becomes a worse bound from below as we add positive terms $\SSS$ to the the quasiconvex translation $\TT$. In particular, if $\TT_0$ is a weak extremal, the best value on the right hand side of \eqref{eq:translation_method} in the class of translations of the form $\TT=\TT_0+\SSS$, where $\SSS$ induces a positive quadratic form, is achieved for $\TT=\TT_0$.

\bibliographystyle{plain}
\bibliography{references}

\end{document}